\documentclass[10pt]{NSP1}
\usepackage{url,floatflt}
\usepackage{helvet,times}
\usepackage{psfig,graphics}
\usepackage{mathptmx,amsmath,amssymb,bm}
\usepackage{float}
\usepackage{fancyhdr}%

\def\no{\noindent}

\begin{document}

\titlefigurecaption{{\large \bf \rm Progress in Fractional Differentiation
and Applications}\\ {\it\small An International Journal}}

\title{Lyapunov inequality for a  boundary value problem involving conformable derivative}
\author{Rabah Khaldi, Assia Guezane-Lakoud}
\institute{ Laboratory of Advanced Materials, Departement of
Mathematics, Badji Mokhtar-Annaba University, P.O. Box 12, 23000
Annaba, Algeria}

\titlerunning{Lyapunov inequality for a  boundary value problem ...}
\authorrunning{R. Khaldi et al. }

\mail{rkhadi@yahoo.fr}

\received{ 3 Apr. 2017} \revised{9 May 2017} \accepted{15 May 2017} \published{.}

\abstracttext{ We consider a boundary value problem involving
conformable derivative of order $\alpha ,$ $1<\alpha <2$ and
Dirichlet conditions. To prove the existence of solutions, we
apply the method of upper and lower solutions together with
Schauder's fixed-point theorem. \ Futhermore, we give the Lyapunov
inequality for the corresponding problem.}
 \keywords{Boundary value problem, Lyapunov inequality, Conformable
derivative, Upper and lower solutions method, Existence of solution.\\[2mm]
\textbf{2010 Mathematics Subject Classification}.  Primary 34B15,
34A08; Secondary 26A33, 34A12.} \maketitle
\section{\; Introduction}

Recently, an interesting derivative called conformable derivative that is
based on a limit form as in the classical derivative was introduced by
Khalil et al. in [20]. Later, this new local derivative is getting more
attention and is improved by Abdeljawad in [1]. The importance of the
conformable derivative is that it has similar properties than the classical
one. Nevertheless, this conformable derivative doesn't satisfy the index law
[18,24] and the zero order derivative property i.e. the zero order
derivative of a differentiable function does not return to the function
itself.

Following this new conformable derivative, several papers have been
presented, in particular some studies about boundary value problems for
conformable differential equations have been the subject of some papers
[3-8,18,24,26]. Furthermore, in [6], Batarfi et al. studied a conformable
differential equation of order $\alpha \in \left( 1,2\right] ,$ with three
point boundary conditions and proved the existence and uniqueness of
solution by using fixed point theorems. In [7], Bayour et al. solved an
initial conformable differential value problem for $\alpha \in \left(
0,1\right) $ by the help of the tube solution method which is a
generalization of the lower and upper solutions method.

In this work, we analyze the existence of solutions for the following
boundary value problem (P)%
\begin{equation*}
T_{\alpha }^{a}u\left( t\right) +f(t,u\left( t\right) )=0,a<t<b,  \tag{1.1}
\end{equation*}

\begin{equation*}
u\left( a\right) =u\left( b\right) =0  \tag{1.2}
\end{equation*}%
where $1<\alpha <2,$ $T_{\alpha }^{a}$ denotes the conformable derivative of
order $\alpha $, $u$ is the unknown function and $f:\left[ a,b\right] \times
\mathbb{R}\rightarrow \mathbb{R}$ is a given function. For this purpose, we
use the method of upper and lower solutions together with Schauder's
fixed-point theorem. The method of lower and upper solutions\ is a powerful
tool in the investigation of the existence of solutions and has been used in
several papers, we refer to [10,13,15,19].

In the case $f(t,u\left( t\right) )=q\left( t\right) u\left( t\right) $, we
prove a new Lyapunov inequality that coincide with the classical one when $%
\alpha =2.$

The classical Lyapunov inequality states that if $q:[a,b]\rightarrow R$ is a
real and continuous function, then a necessary condition for the boundary
value problem

\begin{equation*}
-u^{\prime \prime }\left( t\right) =q\left( t\right) u\left( t\right) ,a<t<b
\end{equation*}%
\begin{equation*}
u\left( a\right) =u\left( b\right) =0
\end{equation*}%
to have nontrivial solutions is that%
\begin{equation*}
\int_{a}^{b}\left\vert q\left( t\right) \right\vert dt\geq \frac{4}{b-a},
\tag{1.3}
\end{equation*}%
see [21]. An equivalent version of the Lyapunov inequality (1.3) was proved
by Borg see [8].%
\begin{equation*}
\int_{a}^{b}\frac{\left\vert u^{\prime \prime }\left( t\right) \right\vert }{%
u\left( t\right) }dt\geq \frac{4}{b-a}.  \tag{1.4}
\end{equation*}%
under the condition $u(t)>0$ for $t\in (a,b)$.

Many authors have extended the Lyapunov inequality by considering a
fractional derivative or a sequential of fractional derivatives instead of
the second derivative in equation (1.1), see [2,9,11-12,14,16-17,21-23,25].
In particular, we cite the paper of Ferreira [12], where he gave the
corresponding Lyapunov type inequalities for both Caputo sequential
fractional differential equation and Riemann-Liouville sequential fractional
differential equation subject to Dirichlet boundary conditions. In [2],
Agarwal et al. obtained Lyapunov type inequalities for mixed nonlinear
Riemann-Liouville fractional differential equations with a forcing term and
Dirichlet boundary conditions. Recently, Guezane-Lakoud et al. [14],
considered a mixed left Riemann--Liouville and right Caputo differential
equation subject to natural conditions and obtained a new Lyapunov type
inequality.

This paper is organized as follows. In Section 2, we present the main
concepts of the conformable derivatives, we give some useful properties and
we prove a property on the extremum of a function for a conformable
derivative. In Section 3, we prove existence of solution to problem (P) by
using the method of upper and lower solutions together with Schauder's
fixed-point theorem. In Section 4, we prove a Lyapunov inequality for
problem (P) in the case $f(t,u\left( t\right) )=q\left( t\right) u\left(
t\right) $.

As far as we know, this work will be the first one that gives the Lyapunov
inequality for conformable differential equations.

\section{\; Preliminaries}

We recall some essential definitions on conformable derivatives that can be
found in [1,20].

Let $n<\alpha <n+1$, and set $\beta =\alpha -n$, for a function $g:\left[
a,\infty \right) \rightarrow \mathbb{R},$ we denote by%
\begin{equation*}
I_{\alpha }^{a}g(t)=\int_{a}^{t}(s-a)^{\alpha -1}g(s)ds,n=0,
\end{equation*}%
and%
\begin{equation*}
I_{\alpha }^{a}g(t)=\frac{1}{n!}\int_{a}^{t}(t-s)^{n}g(s)d\beta
(s,a)=\frac{1}{n!}\int_{a}^{t}(t-s)^{n}(s-a)^{\beta
-1}g(s)ds,n\geq 1.
\end{equation*}

\begin{remark}

\textbf{\ } Notice that, since $0<\beta <1$, then $I_{\alpha
}^{a}g(t)$ is the
Lebesgue-stieltjes integral of the function $(t-s)^{n}g(s)$ on $\left[ a,t%
\right] $ and $d\beta (s,a)=(s-a)^{\beta -1}ds$ is an absolutely continuous
measure with respect to the Lebesgue measure on the real line, generated by
the absolutely continuous function $(t-a)^{\beta }$ and the weight function $%
(s-a)^{\beta -1}\in L_{1}\left[ a,b\right] $ is its Radon-Nikodym derivative
according to the Lebesgue measure.

\end{remark}

The conformable derivative of order\ $0<\alpha <1,$ of a function $g:\left[
a,\infty \right) \rightarrow \mathbb{R}$\ is defined by%
\begin{equation*}
T_{\alpha }^{a}g(t)=\underset{\epsilon \rightarrow 0}{\lim }\frac{g\left(
t+\varepsilon \left( t-a\right) ^{1-\alpha }\right) -g\left( t\right) }{%
\varepsilon },t>a.
\end{equation*}%
If $T_{\alpha }^{a}g(t)$ exists on $(a,b),$ $b>a$ and$\ \underset{%
t\rightarrow a^{+}}{\lim }T_{\alpha }^{a}g(t)$ exists, then we define $%
T_{\alpha }^{a}g(a)=\underset{t\rightarrow a^{+}}{\lim }T_{\alpha }^{a}g(t).$

The conformable derivative of order\ $n<\alpha <n+1$ of a function $g:\left[
a,\infty \right) \rightarrow \mathbb{R}$,\ when $g^{\left( n\right) }$
exists, is defined by%
\begin{equation*}
T_{\alpha }^{a}g(t)=T_{\beta }^{a}g^{\left( n\right) }(t),
\end{equation*}%
where $\beta =\alpha -n\in \left( 0,1\right) .$

For the properties of the conformable derivative, we mention the following:

Let $n<\alpha <n+1$ and $g$ be an $\left( n+1\right) $-differentiable at $t>a
$, then we have

\begin{equation*}
T_{\alpha }^{a}g\left( t\right) =\left( t-a\right) ^{n+1-\alpha
}g^{\left( n+1\right) }\left( t\right)   \tag{2.1}
\end{equation*}

and%
\begin{equation*}
I_{\alpha }^{a}T_{\alpha }^{a}g\left( t\right) =g\left( t\right)
-\sum_{k=0}^{n}\frac{g^{\left( k\right) }\left( a\right) \left( t-a\right)
^{k}}{k!}.
\end{equation*}

\begin{remark}

 \begin{itemize} \item  For $0<\alpha <1$, using (2.1) it follows
that, if a function $g$\ is
differentiable at $t>a$, then one has%
\begin{equation*}
\underset{\alpha \rightarrow 1}{\lim }T_{\alpha }^{a}g(t)=g^{\prime }(t)
\end{equation*}%
and%
\begin{equation*}
\underset{\alpha \rightarrow 0}{\lim }T_{\alpha }^{a}g(t)=(t-a)g^{\prime
}(t),
\end{equation*}%
i.e. the zero order derivative of a differentiable function does not return
to the function itself.

\item Let $n<\alpha <n+1,$ if $g$ is $\left( n+1\right) $-differentiable on $%
(a,b),$ $b>a$ and $\underset{t\rightarrow a^{+}}{\lim }g^{\left( n+1\right) }
$ exists, then from (2.1), we get $T_{\alpha }^{a}g\left( a\right) =\underset%
{t\rightarrow a^{+}}{\lim }T_{\alpha }^{a}g(t)=0.$

\item Let $n<\alpha <n+1,$ if $g$ is $\left( n+1\right) $-differentiable at $%
t>a$, then we can show that $T_{\alpha }^{a}g\left( t\right) =T_{\alpha
-k}^{a}g^{\left( k\right) }\left( t\right) $ for all positive integer $%
k<\alpha .$
\end{itemize}
\end{remark}

Similarly to the classical case, we give a property on the extremum of a
function that has a conformable derivative:

\begin{proposition}
 \textbf{\ } Let $1<\alpha <2,$ if a function $g$ $\in C^{1}\left[ a,b\right] $ attains a
global maximum (respectively minimum) at some point $\xi \in \left(
a,b\right) $, then $T_{\alpha }^{a}g\left( \xi \right) \leq 0$ (respectively
$T_{\alpha }^{a}g\left( \xi \right) \geq 0$).
\end{proposition}

\begin{proof}
\textbf{\ } The result follows from the fact that
\begin{equation*}
T_{\alpha }^{a}g\left( \xi \right) =T_{\alpha -1}^{a}g^{\prime }\left( \xi
\right) =\underset{\varepsilon \rightarrow 0}{\lim }\dfrac{g^{\prime }\left(
\xi +\varepsilon \left( \xi -a\right) ^{2-\alpha }\right) }{\varepsilon }.
\end{equation*}
\end{proof}

\section{\; Existence of solutions}

Let $AC^{2}\left[ a,b\right] =\left\{ u\in C^{1}\left[ a,b\right] ,u^{\prime
}\in AC\left[ a,b\right] \right\} ,$ where $AC\left[ a,b\right] $ is the
space of absolutely continuous functions on $\left[ a,b\right] .$ Denote $%
L^{1}\left( \left[ a,b\right] ,\rho (s)ds\right) $ the Banach space of
Lebesgue integrable functions on $\left[ a,b\right] $ with respect to the
positive weight function $\rho (s)=\left( s-a\right) ^{\alpha -2}\in L^{1}%
\left[ a,b\right] ,$ $1<\alpha <2.$

To prove the existence of solutions for problem (P), we use the lower and
upper solutions method, we need the following definition of lower and upper
solutions for problem (P).

\begin{definition}
\textbf{\ } The functions $\underline{\sigma }$, $\overline{\sigma }\in AC^{2}\left[ a,b%
\right] $ are called lower and upper solutions of problem (P) respectively,
if

a) $T_{\alpha }^{a}\underline{\sigma }\left( t\right) +f\left( t,\underline{%
\sigma }\left( t\right) \right) \geq 0,$for all $t\in \lbrack a,b],$

$\underline{\sigma }\left( a\right) \leq 0,$ $\underline{\sigma }(b)\leq 0.$

b) $T_{\alpha }^{a}\overline{\sigma }\left( t\right) +f\left( t,\overline{%
\sigma }\left( t\right) \right) \leq 0,$ for all $t\in \lbrack a,b],$

$\overline{\sigma }\left( a\right) \geq 0,$ $\overline{\sigma }(b)\geq 0.$
\end{definition}

Next, we solve the following linear boundary value problem.

\begin{lemma}
\textbf{\ } Assume that $y\in C\left[ a,b\right] $, then the
following linear boundary
value problem%
\begin{equation*}
T_{\alpha }^{a}u\left( t\right) +y\left( t\right) =0, a<t<b,\tag{3.1}
\end{equation*}
\begin{equation*}
u\left( a\right) =u\left( b\right) =0,
\end{equation*}%
has a unique solution given by%
\begin{equation*}
u\left( t\right) =\int_{a}^{b}G\left( t,s\right) y\left( s\right)
\rho \left( s\right) ds,  \tag{3.2}
\end{equation*}%
where%
\begin{equation*}
G\left( t,s\right) =\frac{1}{\left( b-a\right) }\left\{
\begin{array}{c}
-\left( b-a\right) \left( t-s\right) +\left( b-s\right) \left( t-a\right)
,a\leq s\leq t\leq b \\
\left( b-s\right) \left( t-a\right) ,a\leq t<s\leq b.%
\end{array}%
\right. .  \tag{3.3}
\end{equation*}
\end{lemma}

\begin{proof}
\textbf{\ } Applying the integral operator $I_{\alpha }^{a},$ to
both sides of the differential equation (3.1), we get
\begin{equation*}
I_{\alpha }^{a}T_{\alpha }^{a}u\left( t\right) +I_{\alpha }^{a}y\left(
t\right) =0,
\end{equation*}%
hence
\begin{equation*}
u\left( t\right) -u\left( a\right) -c\left( t-a\right) +I_{\alpha
}^{a}y\left( t\right) =0.
\end{equation*}%
Since $u\left( a\right) =0,$ then%
\begin{equation*}
u\left( t\right) =-I_{\alpha }^{a}y\left( t\right) +c\left( t-a\right) .
\tag{3.4}
\end{equation*}%
From $u\left( b\right) =0,$ we get%
\begin{equation*}
c=\frac{1}{\left( b-a\right) }I_{\alpha }^{a}y\left( t\right) \left\vert
_{t=b}\right.
\end{equation*}%
Substituting $c$ by its value in (3.4), it yields%
\begin{eqnarray*}
u\left( t\right)  &=&-I_{\alpha }^{a}y\left( t\right) +\frac{\left(
t-a\right) }{\left( b-a\right) }I_{\alpha }^{a}y\left( t\right) \left\vert
_{t=b}\right.  \\
&=&-\int_{a}^{t}\left( t-s\right) \left( s-a\right) ^{\alpha -2}y\left(
s\right) ds+\frac{\left( t-a\right) }{\left( b-a\right) }\int_{a}^{b}\left(
b-s\right) \left( s-a\right) ^{\alpha -2}y\left( s\right) ds \\
&=&\int_{a}^{b}G\left( t,s\right) y\left( s\right) \rho (s)ds,
\end{eqnarray*}%
where the Green function $G$ is given in (3.3).
\end{proof}

\begin{lemma}
\textbf{\ } The Green function $G$ is nonnegative, continuous and satisfies%
\begin{equation*}
0\leq G\left( t,s\right) \leq b-a,\forall s,t\in \left[ a,b\right] .
\tag{3.5}
\end{equation*}
\end{lemma}

Now we give the main result on the existence of solutions for the nonlinear
problem (P).

\begin{theorem}
\textbf{\ } Let $\underline{\sigma }$ and $\overline{\sigma }$ be
the lower and upper solutions of (P) such that $\underline{\sigma
}\leq \overline{\sigma },$ define $E=\{\left( t,x\right) \in
\left[ a,b\right] \times
\mathbb{R}
,\underline{\sigma }\left( t\right) \leq x\leq \overline{\sigma }\left(
t\right) \}$ and assume that $f(t,x)$ is continuous on $E$. Then the problem
(P) has at least one solution $u\in AC^{2}\left( \left[ a,b\right] \right) $
such that
\begin{equation*}
\underline{\sigma }\left( t\right) \leq u\left( t\right) \leq \overline{%
\sigma }\left( t\right) ,a<t<b.
\end{equation*}
\end{theorem}

\begin{proof}
\textbf{\ } Define the modified problem%
\begin{equation*}
\left( MP\right) \left\{
\begin{array}{c}
T_{\alpha }^{a}u\left( t\right) +F\left( t,u(t)\right) =0,a<t<b, \\
u\left( a\right) =u\left( b\right) =0,%
\end{array}%
\right.
\end{equation*}%
where%
\begin{equation*}
F(t,x)=\left\{
\begin{array}{c}
f(t,\overline{\sigma }\left( t\right) )+\frac{\overline{\sigma }\left(
t\right) -x}{x-\overline{\sigma }\left( t\right) +1},\text{\ \ \ \ \ \ \ \ \
\ \ \ \ for }x>\overline{\sigma }\left( t\right) , \\
f(t,x),\text{ \ \ \ \ \ \ \ \ \ \ \ \ \ \ \ \ \ \ \ \ \ \ \ for }\underline{%
\sigma }\left( t\right) \leq x\leq \overline{\sigma }\left( t\right) , \\
f(t,\underline{\sigma }\left( t\right) )+\frac{\underline{\sigma }\left(
t\right) -x}{\underline{\sigma }\left( t\right) -x+1},\text{ \ \ \ \ \ \ \ \
\ \ \ for \ }x<\underline{\sigma }\left( t\right) .%
\end{array}%
\right.
\end{equation*}%
The function $F\left( t,x\right) $ is called a modification of $f(t,x)$
associated with the coupled of lower and upper solutions $\underline{\sigma }
$ and $\overline{\sigma }$. It follows from the definition of $F$ that $%
F(t,x)$ is continuous and $\left\vert F(t,x)\right\vert \leq M$\ on $\left[
a,b\right] \times
\mathbb{R}
$, with $M=M_{0}+1$ where
\begin{equation*}
M_{0}=\max \left\{ \left\vert f(t,x)\right\vert ,(t,x)\in E\right\} .
\end{equation*}%
Define the operator $A$ on $X=C\left[ a,b\right] ,$ by
\begin{equation*}
Au\left( t\right) =\int_{a}^{b}G\left( t,s\right) \left( s-a\right) ^{\alpha
-2}F\left( s,u(s)\right) ds,a\leq t\leq b.
\end{equation*}%
Set $\Omega =\{u\in C\left[ a,b\right] ,\left\vert u\left( t\right)
\right\vert \leq M\frac{\left( b-a\right) ^{\alpha +1}}{\alpha -1},a\leq
t\leq b\}.$ We will show that $A\left( \Omega \right) $ is uniformly
bounded. Let $u\in \Omega $, then, using (3.5), we get
\begin{equation*}
\left\vert Au\left( t\right) \right\vert \leq \int_{a}^{b}G\left( t,s\right)
\left( s-a\right) ^{\alpha -2}\left\vert F\left( s,u(s)\right) \right\vert
ds\leq M\frac{\left( b-a\right) ^{\alpha }}{\alpha -1},
\end{equation*}%
consequently $A\left( \Omega \right) $ is uniformly bounded and $A\left(
\Omega \right) \subset \Omega $.

Now we prove that $A\left( \Omega \right) $ is equicontinuous. For $a\leq
t_{1}<t_{2}\leq b,$ we have
\begin{eqnarray*}
&&\left\vert Au\left( t_{1}\right) -Au\left( t_{2}\right) \right\vert \\
&\leq &\left( t_{2}-t_{1}\right) M\int_{a}^{t_{1}}\left( s-a\right) ^{\alpha
-2}ds+M\int_{t_{1}}^{t_{2}}\left( t_{2}-s\right) \left( s-a\right) ^{\alpha
-2}ds \\
&&+\frac{\left( t_{2}-t_{1}\right) M}{\left( b-a\right) }\int_{a}^{b}\left(
b-s\right) \left( s-a\right) ^{\alpha -2}ds \\
&\leq &\frac{M}{\alpha -1}\left[ \left( t_{2}-t_{1}\right) \left( \left(
t_{1}-a\right) ^{\alpha -1}+\left( b-a\right) ^{\alpha -2}\right)
+t_{2}\left( t_{2}-t_{1}\right) ^{\alpha -1}\right] \rightarrow 0,
\end{eqnarray*}%
when $t_{1}\rightarrow t_{2}.$ Hence, $A\left( \Omega \right) $ is
equicontinuous. Thanks to Arzela-Ascoli's theorem we get that $A$ is
completely continuous. Moreover, by Schauder fixed point theorem we conclude
that $A$ has a fixed point $u\in \Omega $ which is a solution of the
modified problem (MP).

Localization of solution. Let us prove that if $u$ is a solution of the
modified problem (MP), it satisfies%
\begin{equation*}
\underline{\sigma }\left( t\right) \leq u\left( t\right) \leq \overline{%
\sigma }\left( t\right) .  \tag{3.6}
\end{equation*}%
Set $w=u-\overline{\sigma }.$\ Assuming the contrary, so there exists $%
t_{0}\in \left[ a,b\right] $ such that
\begin{equation*}
\underset{t\in \left[ a,b\right] }{\max }w\left( t\right) =w\left(
t_{0}\right) >0
\end{equation*}%
therefore, we have some cases to consider such as the following:

\textbf{Case 1}: If $t_{0}\in \left( a,b\right) ,$ then from
Proposition 1 it yields, $T_{\alpha }^{a}w\left( t_{0}\right) \leq
0.$ Using the fact that $\overline{\sigma }$ is an upper solution
for problem (P), we get
\begin{eqnarray*}
T_{\alpha }^{a}w\left( t_{0}\right)  &=&T_{\alpha }^{a}u\left( t_{0}\right)
-T_{\alpha }^{a}\overline{\sigma }\left( t_{0}\right)  \\
&=&-f(t,\overline{\sigma }\left( t_{0}\right) )-\frac{\overline{\sigma }%
\left( t_{0}\right) -u\left( t_{0}\right) }{u\left( t_{0}\right) -\overline{%
\sigma }\left( t_{0}\right) +1}-T_{\alpha }^{a}\overline{\sigma }\left(
t_{0}\right) >0,
\end{eqnarray*}%
that leads to a contradiction, thus the maximum of $w$ is not achieved\ at
the point $t_{0}\in \left( a,b\right) $.

\textbf{Case 2:} If $t_{0}=a$, we obtain%
\begin{equation*}
w(a)=u\left( a\right) -\overline{\sigma }\left( a\right) >0.
\end{equation*}
On the other hand, since $u$ is solution, then $u\left( a\right) =0$ and
consequently $\overline{\sigma }\left( a\right) <0,$ which contradicts the
fact that $\overline{\sigma }$ is an upper solution of problem (P).

\textbf{Case 3:} If $t_{0}=b$, we obtain a contradiction as in the second
case.
\end{proof}

Applying similar reasoning, we prove that $\underline{\sigma }\left(
t\right) \leq u\left( t\right) ,$ $\forall t\in \left[ a,b\right] .$ Finally
from (3.6) we conclude that $u$ is a solution of problem (P). The proof is
completed.

\section{\; Lyapunov inequality}
 Let $f(t,u\left( t\right) )=q\left( t\right) u\left( t\right) ,$
then
problem (P) becomes%
\begin{equation*}
T_{\alpha }^{a}u\left( t\right) +q\left( t\right) u\left( t\right) =0,a<t<b,
\tag{4.1}
\end{equation*}%
\begin{equation*}
u\left( a\right) =u\left( b\right) =0,
\end{equation*}%
that we denote by (P1). Now we are ready to give the Lyapunov inequality for
problem (P1).

\begin{theorem}
\textbf{\ } Let $q\in C\left( \left[ a,b\right] \right) $. If the
boundary value problem (P1) has a solution $u\in AC^{2}\left(
\left[ a,b\right] \right) $ such that
$u(t)\neq 0$ a.e. on $\left( a,b\right) $, then%
\begin{equation*}
\int_{a}^{b}\left\vert q\left( s\right) \right\vert \rho (s)ds\geq \frac{4}{%
b-a}.  \tag{4.2}
\end{equation*}
\end{theorem}

\begin{proof}
\textbf{\ } Let $u\in AC^{2}\left( \left[ a,b\right] \right) $ be
a solution of problem (P1) such that $u(t)\neq 0$ a.e. on $\left(
a,b\right) $, then from equation (4.1), we can write
\begin{equation*}
\left\vert q\left( t\right) \right\vert =\left\vert \frac{T_{\alpha
}^{a}u\left( t\right) }{u(t)}\right\vert ,  \tag{4.3}
\end{equation*}%
\ a.e. on $\left( a,b\right) .$ Applying the integral operator $I_{\alpha
-1}^{a}$ to both sides of the differential equation (4.3) and following the
same ideas as in [8], we get for all $a<c<d<b$
\begin{eqnarray*}
\left( I_{\alpha -1}^{a}\left\vert q\right\vert \right) \left( t\right)
\left\vert _{t=b}\right. &=&\int_{a}^{b}\left( s-a\right) ^{\alpha
-2}\left\vert q\left( s\right) \right\vert ds \\
&=&\int_{a}^{b}\left( s-a\right) ^{\alpha -2}\left\vert \frac{T_{\alpha
}^{a}u\left( s\right) }{u(s)}\right\vert ds \\
&\geq &\left( \left\Vert u\right\Vert \right) ^{-1}\int_{a}^{b}\left(
s-a\right) ^{\alpha -2}\left\vert \left( s-a\right) ^{2-\alpha }u^{\prime
\prime }\left( s\right) \right\vert ds \\
&\geq &\left( \left\Vert u\right\Vert \right) ^{-1}\int_{c}^{d}\left\vert
u^{\prime \prime }\left( s\right) \right\vert ds.
\end{eqnarray*}%
Since the function $u^{\prime }$ is absolutely continuous on $\left[ a,b%
\right] ,$ it yields%
\begin{equation*}
\left( I_{\alpha -1}^{a}\left\vert q\right\vert \right) \left( t\right)
\left\vert _{t=b}\right. \geq \left( \left\Vert u\right\Vert \right)
^{-1}\left\vert u^{\prime }\left( d\right) -u^{\prime }\left( c\right)
\right\vert ,
\end{equation*}%
where $\left\vert \left\vert u\right\vert \right\vert =\underset{t\in \left[
a,b\right] }{\max }\left\vert u\left( t\right) \right\vert $. Let $%
\left\vert \left\vert u\right\vert \right\vert =u\left( \xi \right) $\ then
the Mean value theorem implies there exist $a<c<\xi $ and $\xi <d<b$ such
that%
\begin{eqnarray*}
\left( I_{\alpha -1}^{a}\left\vert q\right\vert \right) \left( t\right)
\left\vert _{t=b}\right. &\geq &\left( \left\Vert u\right\Vert \right)
^{-1}\left\vert \frac{u\left( b\right) -u\left( \xi \right) }{b-\xi }-\frac{%
u\left( \xi \right) -u\left( a\right) }{\xi -a}\right\vert \\
&=&\frac{1}{b-\xi }+\frac{1}{\xi -a}.
\end{eqnarray*}%
Finally thanks to the harmonic mean inequality, we get (4.2).
\end{proof}

\begin{remark}
\textbf{\ } Note that if $\alpha \rightarrow 2$, then we get the
classical Lyapunov inequality (1.3).
\end{remark}

\section*{\; Acknowledgement}
The authors are grateful to the anonymous referees for their valuable
comments and specially grateful to The Editor-in-Chief Prof. Dumitru
Baleanu, for his comments and suggestions that improved this paper.

\emergencystretch=\hsize

\begin{center}
\rule{6 cm}{0.02 cm}
\end{center}

\end{document}